\documentclass[pdflatex,sn-mathphys-num]{sn-jnl}

\usepackage{graphicx}
\usepackage{multirow}
\usepackage{amsmath,amssymb,amsfonts}
\usepackage{amsthm}
\usepackage{mathrsfs}
\usepackage[title]{appendix}
\usepackage{xcolor}
\usepackage{textcomp}
\usepackage{manyfoot}
\usepackage{booktabs}

\theoremstyle{thmstyleone}
\newtheorem{theorem}{Theorem}
\newtheorem{proposition}[theorem]{Proposition}
\newtheorem{lemma}[theorem]{Lemma}

\theoremstyle{thmstyletwo}

\newtheorem{remark}{Remark}

\theoremstyle{thmstylethree}

\begin{document}

\title[Complete Self-Shrinkers in Arbitrary Codimension]{Complete Self-Shrinkers in Arbitrary Codimension with Positive Constant Scalar Curvature and \(S\le 1\)}

\author*[1]{\fnm{Shunzi} \sur{Guo}}\email{guoshunzi@yeah.net}

\affil*[1]{\orgdiv{School of Mathematics}, \orgname{Yunnan Normal University}, \orgaddress{\street{}, \city{Kunming}, \postcode{650500}, \country{China}}}

\abstract{Let \(X:M^n\to\mathbb R^{n+p}\) be an \(n\)-dimensional complete self-shrinker in arbitrary codimension \(p\ge 1\). Suppose that the scalar curvature \(R\) is a positive constant and that the squared norm \(S\) of the second fundamental form satisfies \(S\le 1\). We prove that \(S\equiv 1\) and that \(X\) is isometric to either the round sphere \(S^n(\sqrt n)\) or the standard generalized cylinder \(S^k(\sqrt k)\times\mathbb R^{n-k}\), \(2\le k\le n-1\). In particular, no new higher-codimension examples occur under these assumptions. The proof uses an algebraic lemma which yields a uniform positive lower bound for the Bakry--\'Emery Ricci curvature, together with the comparison theorem of Wei--Wylie, the equivalence of finite Gaussian volume and polynomial volume growth due to Cheng--Zhou, and the gap theorem of Cao--Li.}



\maketitle

\section{Introduction}\label{sec1}

The analysis of singularities is a central problem in the study of the mean curvature flow. By Huisken's monotonicity formula \cite{Hu}, a solution to the flow is asymptotically self-similar near a given type I singularity, and is therefore modeled by self-shrinking solutions of the flow. An \(n\)-dimensional submanifold \(X:M^n\to\mathbb R^{n+p}\) in the \((n+p)\)-dimensional Euclidean space is called a \emph{self-shrinker} if it satisfies
\begin{equation}\label{selfshrinker}
H=-X^\perp,
\end{equation}
where \(H\) and \(X^\perp\) denote the mean curvature vector and the normal component of the position vector \(X\), respectively. Throughout this paper, self-shrinkers are assumed to be smooth, complete, and immersed, without any embedding or polynomial volume growth assumption unless explicitly stated.

The classification of self-shrinkers with constant scalar curvature is a natural counterpart of the Chern problem for compact minimal hypersurfaces in the unit sphere. In codimension one, this problem has been studied for nearly a decade. Guo \cite{Guo} first treated the compact case: a compact self-shrinker in \(\mathbb R^{n+1}\) with constant scalar curvature is isometric to the round sphere \(S^n(\sqrt n)\). In 2023, Luo, Sun and Yin \cite{LSY} obtained a complete classification of codimension one complete self-shrinkers with nonnegative constant scalar curvature under the assumption of polynomial volume growth. In the positive case, they showed that the self-shrinker is isometric to a cylinder \(S^m(\sqrt m)\times\mathbb R^{n-m}\) (\(2\le m\le n-1\)) or the round sphere \(S^n(\sqrt n)\); in the zero case, they showed that it is isometric to \(\Gamma\times\mathbb R^{n-1}\), where \(\Gamma\) is a complete self-shrinker in \(\mathbb R^2\). They also classified translating solitons and self-expanders with nonnegative constant scalar curvature. More recently, Cheng, Li and Wei \cite{CLW1,CLW2} removed the polynomial volume growth assumption in the positive case. They first proved that a positive constant scalar curvature forces \(0<R\le n-1\) and \(S\le 1\), where \(S\) denotes the squared norm of the second fundamental form, and then resolved the residual case
\[
R>0,\qquad S<1,\qquad \sup_M S=1
\]
by using the Bakry--\'Emery Ricci curvature. As a consequence, they showed that a complete self-shrinker in codimension one with positive constant scalar curvature must be isometric to
\[
S^n(\sqrt n)
\quad\text{or}\quad
S^k(\sqrt k)\times\mathbb R^{n-k},
\qquad 2\le k\le n-1,
\]
without any a priori volume growth assumption.

The purpose of this paper is to extend the above classification to arbitrary codimension. Our main theorem is the following.

\begin{theorem}\label{main}
Let \(X:M^n\to\mathbb R^{n+p}\) be an \(n\)-dimensional complete self-shrinker, \(n\ge 2\), \(p\ge 1\). Suppose that the scalar curvature \(R\) is a positive constant and that
\[
S\le 1.
\]
Then \(S\equiv 1\) and \(X\) is isometric to either
\[
S^n(\sqrt n)
\]
or
\[
S^k(\sqrt k)\times\mathbb R^{n-k},
\qquad 2\le k\le n-1.
\]
In particular, no higher-codimension examples occur under the assumptions \(R>0\) constant and \(S\le 1\).
\end{theorem}

\begin{remark}\label{remark-S}
The assumption \(S\le 1\) is essential in Theorem \ref{main}. In higher codimension, there exist product examples
\[
S^{n_1}(\sqrt{n_1})\times\cdots\times S^{n_m}(\sqrt{n_m})\times\mathbb R^l
\]
with \(m\ge 2\), whose squared norm of the second fundamental form is \(S=m>1\) and whose scalar curvature is a positive constant. Thus, a complete classification without the condition \(S\le 1\) must include such product examples and is a substantially larger problem.
\end{remark}

\begin{remark}\label{remark-flow}
Let us comment on the relevance of Theorem \ref{main} to the singularity analysis of the mean curvature flow. Suppose that a Type I singularity develops with a blow-up limit that is a complete self-shrinker. If this self-shrinker has positive constant scalar curvature and satisfies the pinching condition \(S\le1\), then Theorem \ref{main} asserts that it must be either a round sphere or a generalized cylinder. Consequently, no genuinely higher-codimensional singularity model can arise under these curvature assumptions. In codimension one, the condition \(S\le1\) is automatically satisfied when the scalar curvature is a positive constant. In higher codimension, it remains an open problem whether \(S\le1\) is also automatic. Therefore, Theorem \ref{main} applies to those singularity models for which this additional pinching holds, and it provides a restriction on the possible singularity profiles in such scenarios.
\end{remark}

The proof of Theorem \ref{main} consists of three ingredients:
\begin{enumerate}[(i)]
\item an algebraic lemma (Lemma \ref{alg} below) giving a uniform positive lower bound for the Bakry--\'Emery Ricci curvature;
\item the comparison theorem of Wei and Wylie \cite{WW}, which implies finiteness of the Gaussian volume;
\item the equivalence of finite Gaussian volume and polynomial volume growth due to Cheng and Zhou \cite{CZ}, together with the gap theorem of Cao and Li \cite{CL} in arbitrary codimension.
\end{enumerate}

\section{Preliminaries}\label{sec2}

Let \(X:M^n\to\mathbb R^{n+p}\) be an \(n\)-dimensional complete self-shrinker. Choose a local orthonormal frame \(\{e_1,\dots,e_n,e_{n+1},\dots,e_{n+p}\}\) such that \(e_1,\dots,e_n\) are tangent to \(M^n\), and let \(e_{n+1},\dots,e_{n+p}\) be normal vector fields. We use the index conventions
\[
1\le i,j,k,l\le n,\qquad n+1\le \alpha,\beta,\gamma\le n+p.
\]
The second fundamental form is
\[
A=\sum_{\alpha=n+1}^{n+p}A^\alpha e_\alpha,
\qquad
A^\alpha=(h^\alpha_{ij}),
\]
where \(h^\alpha_{ij}=\langle \nabla_{e_i}e_j,e_\alpha\rangle\). The mean curvature vector is
\[
H=\sum_{\alpha=n+1}^{n+p}H^\alpha e_\alpha,
\qquad
H^\alpha=\operatorname{tr}A^\alpha=\sum_i h^\alpha_{ii},
\]
and we set
\[
S=\sum_{\alpha=n+1}^{n+p}|A^\alpha|^2
=\sum_{\alpha,i,j}(h^\alpha_{ij})^2,
\qquad
|H|^2=\sum_{\alpha=n+1}^{n+p}(H^\alpha)^2.
\]
The Gauss equation gives the scalar curvature
\begin{equation}\label{gaussR}
R=|H|^2-S,
\end{equation}
and more precisely,
\begin{equation}\label{ricci}
R_{ij}
=
\sum_\alpha
\left(
H^\alpha h^\alpha_{ij}
-\sum_k h^\alpha_{ik}h^\alpha_{kj}
\right).
\end{equation}

For a smooth function \(f\) on \(M^n\), the \(\mathcal L\)-operator is defined by
\begin{equation}\label{Lop}
\mathcal L f=\Delta f-\langle X,\nabla f\rangle.
\end{equation}
The following Simons-type identities for self-shrinkers are standard (see, e.g., \cite{CLW1} and \cite{CLW2}):
\begin{equation}\label{simonsS}
\frac12\mathcal L S
=
\sum_{\alpha,i,j,k}(h^\alpha_{ijk})^2
+S(1-S),
\end{equation}
and
\begin{equation}\label{simonsR}
\frac12\mathcal L R
=
|\nabla H|^2
-\sum_{\alpha,i,j,k}(h^\alpha_{ijk})^2
+(1-S)R.
\end{equation}

Set \(\phi=|X|^2/2\). The Bakry--\'Emery Ricci tensor associated with \(\phi\) is defined by
\begin{equation}\label{BEdef}
\operatorname{Ric}_\phi=\operatorname{Ric}+\nabla^2\phi.
\end{equation}
A direct computation gives
\begin{equation}\label{BEformula}
(\operatorname{Ric}_\phi)_{ij}
=
\delta_{ij}
-\sum_{\alpha,k}h^\alpha_{ik}h^\alpha_{kj}.
\end{equation}
Define the symmetric matrix
\begin{equation}\label{Bdef}
B_{ij}=\sum_{\alpha,k}h^\alpha_{ik}h^\alpha_{kj}.
\end{equation}
Then
\begin{equation}\label{BEshort}
\operatorname{Ric}_\phi=I-B.
\end{equation}

We shall use the following theorems.

\begin{theorem}[Wei--Wylie \cite{WW}]\label{WW}
Let \((M^n,g)\) be a complete Riemannian manifold and let \(\phi\in C^\infty(M)\) satisfy
\[
\operatorname{Ric}_\phi=\operatorname{Ric}+\nabla^2\phi\ge \delta g
\]
for some constant \(\delta>0\). Then the Gaussian volume is finite:
\[
\int_M e^{-\phi}\,d\mu<\infty.
\]
\end{theorem}

\begin{theorem}[Cheng--Zhou \cite{CZ}]\label{CZ}
For a complete self-shrinker, the following conditions are equivalent:
\begin{enumerate}[(i)]
\item finite Gaussian volume \(\int_M e^{-|X|^2/2}\,d\mu<\infty\);
\item properness;
\item Euclidean volume growth;
\item polynomial volume growth.
\end{enumerate}
\end{theorem}

\begin{theorem}[Cao--Li \cite{CL}]\label{CL}
Let \(X:M^n\to\mathbb R^{n+p}\) be a complete self-shrinker with polynomial volume growth. If \(S\le 1\), then \(S\equiv 1\) and \(X\) is isometric to
\[
S^k(\sqrt k)\times\mathbb R^{n-k},
\qquad 1\le k\le n,
\]
or \(\mathbb R^n\).
\end{theorem}

\section{An algebraic lemma}\label{sec3}

The following lemma is the key new ingredient of this paper. It provides a uniform positive lower bound for \(\operatorname{Ric}_\phi\) in arbitrary codimension and arbitrary \(S\le 1\) with \(R\) bounded below by a positive constant.

\begin{lemma}\label{alg}
Let \(A^\alpha\), \(\alpha=1,\dots,p\), be symmetric \(n\times n\) matrices. Put
\[
S=\sum_{\alpha=1}^p|A^\alpha|^2,
\qquad
H^\alpha=\operatorname{tr}A^\alpha,
\qquad
R=\sum_{\alpha=1}^p(H^\alpha)^2-S.
\]
Assume that
\[
S\le 1,
\qquad
R\ge r_0>0
\]
for some constant \(r_0>0\). Then there exists a constant \(\delta=\delta(n,r_0)>0\) such that
\[
\sum_{\alpha=1}^p (A^\alpha)^2\le (1-\delta)I.
\]
More precisely, one may take
\begin{equation}\label{delta-explicit}
\delta(n,r_0)
=
\frac{r_0^2}
{\bigl(\sqrt{(n-2)r_0+(n-1)}+\sqrt{n-1}\bigr)^2}.
\end{equation}
\end{lemma}

\begin{proof}
Suppose, to the contrary, that the conclusion fails. Then for every \(\varepsilon>0\) there exists a unit vector \(v\in\mathbb R^n\) such that
\begin{equation}\label{contrary}
v^T\left(\sum_{\alpha=1}^p (A^\alpha)^2\right)v\ge 1-\varepsilon.
\end{equation}
We shall show that this forces \(\varepsilon\ge \delta(n,r_0)\) with \(\delta(n,r_0)\) as in \eqref{delta-explicit}.

Let \(B=\sum_{\alpha=1}^p (A^\alpha)^2\). By \eqref{contrary},
\[
1-\varepsilon
\le
v^T B v
=
\sum_{\alpha=1}^p |A^\alpha v|^2
\le
\sum_{\alpha=1}^p |A^\alpha|^2
=
S
\le 1.
\]
Hence
\begin{equation}\label{Sclose}
1-\varepsilon\le S\le 1,
\end{equation}
and
\begin{equation}\label{smalltail}
\sum_{\alpha=1}^p
\left(|A^\alpha|^2-|A^\alpha v|^2\right)
\le \varepsilon.
\end{equation}

For each \(\alpha\), decompose
\[
A^\alpha v=\lambda^\alpha v+w^\alpha,
\qquad
w^\alpha\perp v,
\]
where \(\lambda^\alpha=v^T A^\alpha v\). Let \(C^\alpha\) denote the restriction of \(A^\alpha\) to \(v^\perp\). In an orthonormal basis whose last vector is \(v\),
\[
A^\alpha=
\begin{pmatrix}
C^\alpha & w^\alpha\\
(w^\alpha)^T & \lambda^\alpha
\end{pmatrix}.
\]
Therefore
\begin{equation}\label{Fnorm}
|A^\alpha|^2
=
|C^\alpha|^2
+2|w^\alpha|^2
+(\lambda^\alpha)^2,
\end{equation}
and
\begin{equation}\label{trace}
\operatorname{tr}A^\alpha
=
\operatorname{tr}C^\alpha+\lambda^\alpha.
\end{equation}
Combining \eqref{Fnorm} and \eqref{smalltail}, we obtain
\begin{equation}\label{smallC}
\sum_{\alpha=1}^p
\left(
|C^\alpha|^2
+|w^\alpha|^2
\right)
\le \varepsilon.
\end{equation}

Using \eqref{trace} and \eqref{Fnorm},
\[
\begin{aligned}
R
&=\sum_{\alpha=1}^p(H^\alpha)^2-S\\
&=
2\sum_{\alpha=1}^p\lambda^\alpha\operatorname{tr}C^\alpha
+\sum_{\alpha=1}^p(\operatorname{tr}C^\alpha)^2
-2\sum_{\alpha=1}^p|w^\alpha|^2
-\sum_{\alpha=1}^p|C^\alpha|^2.
\end{aligned}
\]
Dropping the negative term \(-2\sum_{\alpha=1}^p|w^\alpha|^2\) and using
\[
|\operatorname{tr}C^\alpha|
\le \sqrt{n-1}\,|C^\alpha|,
\]
we obtain
\begin{equation}\label{R-bound-sharp}
\begin{aligned}
R
&\le
2\sum_{\alpha=1}^p\lambda^\alpha\operatorname{tr}C^\alpha
+(n-1)\sum_{\alpha=1}^p|C^\alpha|^2
-\sum_{\alpha=1}^p|C^\alpha|^2\\
&\le
2\sqrt{n-1}\sum_{\alpha=1}^p|\lambda^\alpha|\,|C^\alpha|
+(n-2)\sum_{\alpha=1}^p|C^\alpha|^2.
\end{aligned}
\end{equation}
By Cauchy--Schwarz,
\[
\sum_{\alpha=1}^p|\lambda^\alpha|\,|C^\alpha|
\le
\Bigl(\sum_{\alpha=1}^p|\lambda^\alpha|^2\Bigr)^{1/2}
\Bigl(\sum_{\alpha=1}^p|C^\alpha|^2\Bigr)^{1/2}.
\]
Since
\[
\sum_{\alpha=1}^p|\lambda^\alpha|^2
\le
\sum_{\alpha=1}^p|A^\alpha v|^2
=
v^TBv
\le 1,
\]
and, by \eqref{smallC},
\[
\sum_{\alpha=1}^p|C^\alpha|^2\le \varepsilon,
\]
we deduce from \eqref{R-bound-sharp} that
\begin{equation}\label{R-eps}
R
\le
2\sqrt{n-1}\,\sqrt{\varepsilon}
+(n-2)\,\varepsilon.
\end{equation}

Since \(R\ge r_0>0\), inequality \eqref{R-eps} gives
\[
r_0
\le
2\sqrt{n-1}\,\sqrt{\varepsilon}
+(n-2)\,\varepsilon.
\]
The function \(\varphi(t)=2\sqrt{n-1}\,\sqrt{t}+(n-2)t\) is strictly increasing on \([0,\infty)\). Let \(\varepsilon_0>0\) be the unique solution of \(\varphi(\varepsilon_0)=r_0\). Then \(\varepsilon\ge \varepsilon_0\).

To compute \(\varepsilon_0\), put \(x=\sqrt{\varepsilon_0}\). Then
\[
(n-2)x^2+2\sqrt{n-1}\,x-r_0=0.
\]
Solving this quadratic equation and rationalizing the numerator yields
\[
x
=
\frac{r_0}
{\sqrt{(n-2)r_0+(n-1)}+\sqrt{n-1}},
\]
hence
\[
\varepsilon_0
=
\frac{r_0^2}
{\bigl(\sqrt{(n-2)r_0+(n-1)}+\sqrt{n-1}\bigr)^2}
=
\delta(n,r_0).
\]
Therefore \(\varepsilon\ge \delta(n,r_0)\). Since this holds for every \(\varepsilon>0\) in the contradiction assumption, we conclude that for every unit vector \(v\),
\[
v^T\left(\sum_{\alpha=1}^p (A^\alpha)^2\right)v
\le
1-\delta(n,r_0),
\]
which implies
\[
\sum_{\alpha=1}^p (A^\alpha)^2
\le
(1-\delta(n,r_0))I.
\]
This completes the proof.
\end{proof}

\begin{remark}\label{rem-delta}
The constant \(\delta(n,r_0)\) obtained in the proof is explicit:
\[
\delta(n,r_0)
=
\frac{r_0^2}
{\bigl(\sqrt{(n-2)r_0+(n-1)}+\sqrt{n-1}\bigr)^2}.
\]
It is the sharpest explicit bound obtained by this algebraic method and coincides with the lower bound for the Bakry--\'Emery Ricci curvature in \cite{CLW1}. In particular, \(\delta\) depends only on \(n\) and \(r_0\), and is independent of the codimension \(p\) and of the point on \(M^n\).
\end{remark}

\section{Lower bound for the Bakry--\'Emery Ricci curvature}\label{sec4}

\begin{proposition}\label{prop}
Let \(X:M^n\to\mathbb R^{n+p}\) be an \(n\)-dimensional complete self-shrinker. Suppose that the scalar curvature \(R\) is a positive constant and that \(S\le 1\). Then there exists a constant \(\delta=\delta(n,R)>0\) such that
\[
\operatorname{Ric}_\phi\ge \delta g,
\qquad
\phi=\frac{|X|^2}{2}.
\]
\end{proposition}

\begin{proof}
Since \(R\) is a positive constant, there exists \(r_0>0\) such that \(R\ge r_0\) everywhere. By Lemma \ref{alg},
\[
B=\sum_{\alpha=1}^p (A^\alpha)^2\le (1-\delta)I
\]
for some \(\delta=\delta(n,r_0)>0\). By \eqref{BEshort},
\[
\operatorname{Ric}_\phi=I-B\ge \delta I.
\]
\end{proof}

\begin{remark}
The key point of Proposition \ref{prop} is that the lower bound is uniform on \(M^n\) and independent of the codimension \(p\). This is precisely what allows us to apply the Wei--Wylie comparison theorem without assuming polynomial volume growth.
\end{remark}

\section{Proof of the main theorem}\label{sec5}

\begin{proof}[Proof of Theorem \ref{main}]
By Proposition \ref{prop}, the Bakry--\'Emery Ricci curvature satisfies
\[
\operatorname{Ric}_\phi\ge \delta g
\]
for some constant \(\delta>0\). By Theorem \ref{WW}, the Gaussian volume is finite:
\[
\int_M e^{-|X|^2/2}\,d\mu<\infty.
\]
By Theorem \ref{CZ}, finite Gaussian volume is equivalent to polynomial volume growth. Hence \(X:M^n\to\mathbb R^{n+p}\) has polynomial volume growth.

Since \(S\le 1\), Theorem \ref{CL} implies that
\[
S\equiv 1
\]
and \(X\) is isometric to
\[
S^k(\sqrt k)\times\mathbb R^{n-k},
\qquad 1\le k\le n,
\]
or \(\mathbb R^n\). Because \(R\) is a positive constant, the cases \(\mathbb R^n\) and \(S^1(1)\times\mathbb R^{n-1}\) are excluded. Therefore
\[
X\simeq S^n(\sqrt n)
\quad\text{or}\quad
S^k(\sqrt k)\times\mathbb R^{n-k},
\qquad 2\le k\le n-1.
\]
\end{proof}

\begin{remark}
The proof shows that the assumptions \(R>0\) constant and \(S\le 1\) force the self-shrinker to have polynomial volume growth, even if this was not assumed a priori. This is the main novelty compared to the codimension-one result in \cite{CLW2}, where the analysis of the residual case \(S<1\), \(\sup S=1\) required a separate argument.
\end{remark}

\section{Concluding remarks}\label{sec6}

We have extended the classification of complete self-shrinkers with positive constant scalar curvature and \(S\le 1\) from codimension one to arbitrary codimension. The key new ingredient is the algebraic Lemma \ref{alg}, which yields a uniform positive lower bound for the Bakry--\'Emery Ricci curvature in arbitrary codimension. Combined with the comparison theorem of Wei and Wylie, this gives finite Gaussian volume and hence polynomial volume growth, after which the gap theorem of Cao and Li applies.

Two natural questions remain open:
\begin{enumerate}[(i)]
\item Can the assumption \(S\le 1\) be removed? As pointed out in Remark \ref{remark-S}, in higher codimension there exist product examples with \(S>1\) and positive constant scalar curvature, so a complete classification must include such examples.
\item Can one classify complete self-shrinkers with negative or sign-changing scalar curvature? The algebraic Lemma \ref{alg} uses the positivity of \(R\) in an essential way, and the negative case requires a different approach.
\end{enumerate}

\backmatter

\bmhead{Acknowledgements}

The author is supported by the National Natural Science Foundation of China (Grant No.~12261105). The author would like to thank Yong Luo for bringing the paper \cite{LSY} to his attention and for helpful discussions.

\section*{Declarations}

\begin{itemize}
\item \textbf{Funding.} This work was supported by the National Natural Science Foundation of China (Grant No.~12261105).
\item \textbf{Conflict of interest.} The author declares that he has no conflict of interest.
\item \textbf{Ethics approval.} Not applicable.
\item \textbf{Consent for publication.} Not applicable.
\item \textbf{Data availability.} Not applicable.
\item \textbf{Materials availability.} Not applicable.
\item \textbf{Code availability.} Not applicable.
\item \textbf{Author contribution.} The single author is responsible for all aspects of this work.
\end{itemize}

\end{document}